\documentclass[a4paper,11pt]{amsart}

\usepackage{amsmath, enumerate, amsfonts, amssymb, amsthm}

\newtheorem{defin}{\bf Def\mbox{}inition}[section]

\newtheorem{prop}[defin]{\bf Proposition}
\newtheorem{lem}[defin]{\bf Lemma}
\newtheorem{cor}[defin]{\bf Corollary}

\newtheorem{rem}[defin]{\bf Remark}

\newtheorem*{clai*}{\bf Claim}
\newtheorem*{rem*}{\bf Remark}
\newtheorem*{nota*}{\bf Notation}
\newtheorem*{prop*}{\bf Proposition}

\newcommand{\C}{\mathbb{C}}
\newcommand{\N}{\mathbb{N}}

\newcommand{\Q}{\mathbb{Q}}
\renewcommand{\k}{\mathbf{k}}
\newcommand{\K}{\mathbf{K}}

\newcommand{\Gal}{\mathrm{Gal}(\K / \k)}

\newcommand{\Up}{\Upsilon}

\renewcommand{\O}{\mathcal{O}} % faisceau des fonctions holomorphes
\newcommand{\Ohat}{\hat{\mathcal{O}}} % Series formelles

\newcommand{\D}{\mathcal{D}} % analytic notations

\newcommand{\Dhat}{\hat{D}} % formal notations

\newcommand{\dxi}{\partial _{x_i}}
\newcommand{\dtj}{\partial _{t_j}}
\newcommand{\dx}[1]{\partial _{x_{#1}}}
\newcommand{\dt}[1]{\partial_{t_{#1}}}
\newcommand{\ddx}{\partial _x}
\newcommand{\ddt}{\partial_t}

\newcommand{\dxsur}[2]{\frac{\partial {#1}}{\partial x_{#2}}}

\newcommand{\B}{\mathcal{B}}
\title{Algorithm for computing local Bernstein-Sato ideals}

\author{Rouchdi BAHLOUL \and Toshinori Oaku}
\address{
Universit\'e de Lyon, Univ. Lyon 1,
Institut Camille Jordan, CNRS UMR 5208,
43 Bd du 11 novembre 1918,
F-69622 Villeurbanne, France}
\email{bahloul@math.univ-lyon1.fr}
\address{
Department of Mathematics,
Tokyo Woman's Christian University,
Suginami-ku,
Tokyo 167-8585,
Japan}
\email{oaku@lab.twcu.ac.jp}

\begin{document}

\begin{abstract}
Given $p$ polynomials of $n$ variables over a field $\k$
of characteristic $0$ and a point $a \in \k^n$, we
propose an algorithm computing the local
Bernstein-Sato ideal at $a$. Moreover with the same
algorithm we compute a constructible
stratification of $\k^n$ such that the local Bernstein-Sato
ideal is constant along each stratum. Finally, we present
non-trivial examples computed with our algorithm.
\end{abstract}

%\subjclass[2000]{}
%\keywords{}

\maketitle

%\tableofcontents

\section*{Introduction}

Let $n$ be a positive integer and $\k$ be a field of
characteristic zero and $a=(a_1, \ldots, a_n)$ be a fixed
point in $\k^n$. Let $x=(x_1, \ldots, x_n)$ be a set
of indeterminates.
In this introduction, $A$ shall be
one of the following rings: the polynomial ring $\k[x]$;
the localization $\k[x]_{a}$ of $\k[x]$ at $a$;
the formal power series ring
$\Ohat_{\k^n, a}=\k[[x-a]]:=\k[[x_1-a_1, \ldots, x_n-a_n]]$;
and when $\k=\C$, the ring $\O_{\C^n,a}$ of germs of
complex analytic functions at $a$.
Denote by $\dxi$ the partial differential operator
$\dxsur{}{i}$ and $D_A=A\langle \dx{1}, \ldots, \dx{n}
\rangle$ the ring of differential operators with
coefficients in $A$.

Let $p\ge 1$ be an integer and let us
consider $f=(f_1,\ldots, f_p) \in A^p$.
Denote by $F$ the product $f_1 \cdots f_p$ and let us
introduce a set of indeterminates $s=(s_1, \ldots, s_p)$
then consider the $A[1/F, s]$-free module
\[\mathcal{L}_A= A[1/F, s] \cdot f^s\]
where $f^s$ is seen as $f_1^{s_1} \cdots f_p^{s_p}$.
The set $\mathcal{L}_A$ is naturally endowed with a
$D_A[s]$-module structure.
Indeed, given $g \in A[1/F,s]$, we have:
\[\dxi \cdot g f^s=(\dxsur{g}{i} + g \sum_{j=1}^p
s_j \dxsur{f_j}{i} f_j^{-1}) f^s.\]
A Bernstein-Sato polynomial (or $b$-function) of $f$
with respect to $A$ is a polynomial $b(s) \in \k[s]$ such that
\[b(s) f^s \in D_A[s] \cdot F f^s.\]
Such $b(s)$ form the ideal $\B_A(f)$, the Bernstein-Sato
ideal of $f$ (with respect to A). When $p=1$, the monic generator (if
it is not zero) is called the Bernstein-Sato polynomial of $f$
(with respect to A).

When $f \in \k[x]^p$, $\B_{\k[x]}(f)$ is called the
global Bernstein-Sato ideal and $\B_{\k[x]_a}(f)$ is called
the local Bernstein-Sato ideal at $a \in \k^n$. It is
well-known that $\B_{\k[x]_a}(f)$ is equal to
$\B_{\Ohat_{\k^n, a}}(f)$ and in the case where $\k=\C$,
it is equal to $\B_{\O_{\C^n, a}}(f)$
(see Brian\c{c}on, Maisonobe \cite{bm90} and the general
arguments in Mebkhout, Narv\`aez-Macarro \cite[4.2]{meb-nar}
for these equalities).
When $f \in \O_{\C^n,a}$, $\B_{\O_{\C^n,a}}(f)$ is called
the analytic Bernstein-Sato ideal of $f$, and it is equal
to $\B_{\Ohat_{\C^n,a}}(f)$.
Finally, when $f \in \Ohat_{\k^n, a}$, we call
$\B_{\Ohat_{\k^n,a}}(f)$ the formal Bernstein-Sato ideal
of $f$ (at $a$).

Let us recall some historical facts.
The global or analytic or local Bernstein-Sato ideal is not
zero (see Bernstein \cite{bernstein} and Sabbah \cite{sabbah}
and also \cite{b-compos}). When $p=1$, the formal
$b$-function is not zero (see Bj\"ork \cite{bjork}).
However for $p\ge 2$, it is still an open question to know
whether $\B_{\Ohat_{\k^n, a }}(f)$ is zero or not.

When $f$ is a polynomial mapping,
a general algorithm for the global Bernstein-Sato ideal
was given by Oaku and Takayama \cite{ot}
(see also Bahloul \cite{b-JSC} and Brian\c{c}on,
Maisonobe \cite{bm02}). When $p=1$, Oaku \cite{oaku} gave
an algorithm for the local $b$-function (see also the recent
work by H. Nakayama \cite{nakayama}).
The first goal of the present paper is, given $f\in \k[x]^p$,
to present an algorithm for computing $\B_{\k[x]_a}(f)$
for $p \ge 1$.
Secondly, by works of Brian\c{c}on and Maisonobe \cite{bm02}
we know that there exists a constructible stratification of
$\k^n$ such that for $a$ running over a given stratum the local
Bernstein-Sato ideal at $a$ is constant. With our
algorithm, we obtain such a stratification.

In section 1 we present all the main results. For the sake of
clarity, all the proofs are postponed to section 2. Section 3
contains non-trivial examples treated with our algorithm.
For example we treat $f=(x^3+y^2,x^2+y^3)$ which was intractable
until today. We also treat $f=(z, x^4+y^4+2zx^2y^2)$ taken from
Brian\c{c}on and Maynadier \cite{bmay}. They proved
that $\B_{\C[x,y,z]_0}(f)$ is not principal by general geometric
arguments without computing it.

\section{Statement of the main results and description
of the algorithm}

In this section we shall recall the main lines of the
(known) algorithm for the global Bernstein-Sato ideal
(\cite{ot} and \cite{bm02}). At the same time we shall
describe our local algorithm emphasizing the common points
and the differences with the global case.

Let us fix a polynomial mapping $f \in \k[x]^p$.
We are interested in $\B_{\k[x]_a}(f)$.
As we recalled, the formal, analytic and local
Bernstein-Sato ideals of $f$ at $a$ are all the same, so we
shall use the notation: $\B_{loc,a}(f)= \B_{\k[x]_a}(f)$. This
will be compared with the global Bernstein-Sato ideal
$\B_{glob}(f)=\B_{\k[x]}(f)$.

Moreover, we shall use the following notations:
$D=\k[x]\langle \ddx \rangle$, $D_a=\k[x]_a\langle \ddx \rangle$,
$\Dhat_a=\Ohat_{\k^n, a}\langle \ddx \rangle$,
and when $\k=\C$, $\D_a=\O_{\C^n, a}\langle \ddx \rangle$.

Following Malgrange \cite{malgrange}, let us introduce
new variables $t_1,\ldots, t_p$ together with the associated
partial derivation operators $\dt{1}, \ldots, \dt{p}$; then
let us consider the ring
$\Dhat_a\langle t, \ddt \rangle=\Dhat_a \otimes_\k
\k[t]\langle \ddt \rangle $
where $t=(t_1,\ldots, t_p)$ and $\ddt=(\dt{1}, \ldots,
\dt{p})$. We then have the subrings
$D\langle t, \ddt \rangle$, $D_a\langle t, \ddt \rangle$
and (when $\k=\C$) $\D_a \langle t, \ddt \rangle$.

The free module $\mathcal{L}_{\Ohat_{\k^n, a}}=
\Ohat_{\k^n,a}[1/F,s] f^s$ has a
$\Dhat_a \langle t, \ddt \rangle$-module structure:
\[t_j \cdot g(s) f^s= g(s_1, \ldots, s_j+1, \ldots, s_p)
f_j f^s\]
and
\[\dt{j} \cdot g(s) f^s= -s_j g(s_1, \ldots, s_j-1,
\ldots, s_p) f_j^{-1} f^s\]
where $g(s) \in \Ohat_{\k^n,a}[1/F, s]$. It also has
a module structure over the subrings of $\Dhat_a \langle t,\ddt
\rangle$. Moreover, we have that $-\dtj t_j$ acts on
$\mathcal{L}_{\Ohat_{\k^n, a}}$ as $s_j$. Thus we shall
identify $s_j$ with $-\dtj t_j$ and the rings
$D[s]$, $D_a[s]$, $\D_a[s]$ and $\Dhat_a[s]$ shall be
seen as subrings of $\Dhat_a \langle t, \ddt \rangle$.

Let us consider the following elements:

\begin{equation}\label{eq:I}
t_j - f_j, \ j=1,\ldots, p\ ; \  \ \dxi+
\sum_{j=1}^p \dxsur{f_j}{i} \dtj , \ i=1, \ldots, n. 
\end{equation}
One can easily check that these elements annihilate $f^s$,
in fact we have:

\begin{lem}\label{lem}
The annihilating ideals\footnote{In a non commutative ring,
ideal shall always mean left ideal.}
$\mathrm{ann}_{\Dhat_a\langle t, \ddt \rangle}(f^s)$,
$\mathrm{ann}_{\D_a\langle t, \ddt \rangle}(f^s)$,
$\mathrm{ann}_{D_a\langle t, \ddt \rangle}(f^s)$,
$\mathrm{ann}_{D\langle t, \ddt \rangle}(f^s)$
are all generated by the elements in (\ref{eq:I}).
\end{lem}

\noindent
Let us introduce the following ideals:
\begin{itemize}
\item
$I=\mathrm{ann}_{D\langle t, \ddt \rangle}(f^s)$,\\
$J=\mathrm{ann}_{D_a\langle t, \ddt \rangle}(f^s)$.
\item
$I_1= I \cap D[s] \subset D[s]$,\\
$J_1=J \cap D_a[s] \subset D_a[s]$.
\item
$I_2 = (I_1 + D[s] \cdot F) \cap \k[x,s] \subset
\k[x_1, \ldots, x_n, s_1, \ldots, s_p]$,\\
$J_2 = (J_1 + D_a[s] \cdot F) \cap \k[x]_a[s] \subset
\k[x_1, \ldots, x_n]_a[ s_1, \ldots, s_p]$.
\item
$I_3=I_2 \cap \k[s] \subset \k[s_1, \ldots, s_p]$,\\
$J_3=J_2 \cap \k[s] \subset \k[s_1, \ldots, s_p]$.
\end{itemize}

Obviously, $I_1$ (resp. $J_1$) is the annihilator of
$f^s$ in $D[s]$ (resp. in $D_a[s]$).

\begin{prop}\label{p1}
\[I_3=\B_{glob}(f) \qquad \text{and}\qquad J_3=\B_{loc,a}(f).\]
\end{prop}

In both cases (global and local) we start with the ``same''
ideals $I$ and $J$ (``same'' means that $J$ and $I$ admit
a common set of generators); then we construct in parallel
the ideals $I_k$ and $J_k$, with $k=1,2,3$, and finally get
the global Bernstein-Sato ideal $I_3=\B_{glob}(f)$
and the local Bernstein-Sato ideal $J_3=\B_{loc,a}(f)$.
It is natural to ask whether $I_k$ and $J_k$ are the
``same'' (in the above sense). Here is the beginning
of the answer.

\begin{prop}\label{p2}
We have:
\begin{itemize}
\item
$J_1= D_a[s] \cdot I_1$,
\item
$J_2=\k[x]_a[s] \cdot I_2$.
\end{itemize}
\end{prop}

This proposition says that the global and the local constructions
coincide (in the above sense) at least up to step $2$.
Going from $I_2$ to $I_3$ consists in a usual elimination.
However, going from $J_2$ to $J_3$ is different. Here it is:

\begin{prop}\label{p3}
Let $\Up$ be an ideal in $\k[x,s]$.
Let $\Up=\Up_1 \cap \cdots \cap \Up_r$
be a primary decomposition of $\Up$.
Let $\sigma_a \subset \{1,\ldots, r\}$ be such that

$i \in \sigma_a \iff a \in V(\Up_i \cap \k[x])$ (here $V( \cdot )$
means ``the zero set of''). Then,
\begin{itemize}
\item
if $\sigma_a =\emptyset$ then
$(\k[x]_a[s] \cdot \Up)\cap \k[s]= \k[s]$,
\item
otherwise: $(\k[x]_a[s] \cdot \Up) \cap \k[s]=
(\bigcap_{i \in \sigma_a} \Up_i) \cap \k[s]$.
\end{itemize}
\end{prop}

Thus applying this to $\Up=I_2$ tells us that combining
a primary decomposition, a finite intersection of ideals and
an elimination of variables, we obtain $J_3$.

The results above say that the computation
of the local Bernstein-Sato ideal coincides with
that of the global Bernstein-Sato ideal
up to a point where everything is commutative.
Computing $I_1$ and $I_2$ is well-known. For $I_1$
one can follow Oaku and Takayama \cite{ot}
or Brian\c{c}on and Maisonobe \cite{bm02} and for
$I_2$ it is a usual elimination of global
variables.

Until now, $a$ was a fixed point in $\k^n$.
In the following result we are concerned with the
behaviour of $\B_{loc,a}(f)$ when $a$ runs over $\k^n$.
Let us apply the previous proposition to a primary
decomposition $\Up_1 \cap \cdots \cap \Up_r$ of $I_2$,
then the following holds.

\begin{cor}\label{cor:strat}
For each subset $\sigma \subset \{1,\ldots,r\}$, let
\begin{itemize}
\item
$W_\sigma=\k^n \setminus
(\bigcup_{i=1}^n V(\Up_i \cap \k[x]))$
if $\sigma= \emptyset$,
\item
$W_\sigma=\bigcap_{i=1}^n V(\Up_i \cap \k[x])$
if $\sigma=\{1,\ldots, n\}$,
\item
$W_\sigma=(\bigcap_{i \in \sigma} V(\Up_i \cap \k[x]))
\setminus (\bigcup_{i \notin \sigma} V(\Up_i \cap \k[x]))$
otherwise.
\end{itemize}
Then $\bigcup_{\sigma} W_\sigma$ is a constructible
stratification of $\k^n$ such that the map $\k^n \ni a
\mapsto \B_{loc,a}(f)$ is constant on each $W_\sigma$.
\end{cor}

As another consequence of our algorithm, we recover this
well-known result:
\begin{cor}\label{cor:glob-loc}
Assume that $\k$ is algebraically closed. Then
\[\B_{glob}(f) = \bigcap_{a \in \k^n} \B_{loc,a}(f).\]
\end{cor}

\begin{rem*}
With a slight generalisation of the construction and
similar proofs one can obtain the following result
when $\k$ is not supposed to be algebraically
closed (see \cite[Prop. 1.4]{bm02}):
\[\B_{glob}(f)
= \bigcap_{m \in \mathrm{SpecMax(\k[x])}} \B_{loc,m}(f),\]
where $\B_{loc,m}(f)$ is the set of $b(s) \in \k[s]$ such that
there exists $a(x) \in \k[x] \smallsetminus m$ such that
$a(x) b(s) f^s \in D[s] F  f^s$.
\end{rem*}

Let us end this section with some additionnal results.

Again, $a$ is a fixed point in $\k^n$.
The first statement in Proposition \ref{p2} says
that the annihilators of $f^s$ in $D[s]$ and
in $D_a[s]$ have a common set of generators.
Similarly we have:

\begin{prop}\label{p4}
\[ \mathrm{ann}_{\Dhat_a[s]}(f^s)= \Dhat_a[s] \cdot \mathrm{ann}_{D[s]}(f^s)
\qquad \text{and} \qquad
\mathrm{ann}_{\Dhat_a}(f^s)= \Dhat_a \cdot \mathrm{ann}_{D}(f^s).\]
\end{prop}

\begin{rem} \
\begin{itemize}
\item
With the same proof we get: if $\k=\C$ then
\[\mathrm{ann}_{\D_a[s]}(f^s)= \D_a[s] \cdot \mathrm{ann}_{D[s]}(f^s)
\qquad \text{and} \qquad
\mathrm{ann}_{\D_a}(f^s)= \D_a \cdot \mathrm{ann}_{D}(f^s).\]
\item
Set $\hat{I}_1=\mathrm{ann}_{\Dhat_a[s]}(f^s)$;
$\hat{I}_2=(\hat{I}_1+ \Dhat_a[s] F) \cap \Ohat_{\k^n,a}[s]$;
$\hat{I}_3=\hat{I}_2 \cap \k[s]$.

With standard bases methods, one can prove that:
$\hat{I}_2=\Ohat_{\k^n, a}[s] \cdot I_2$.

By using the faithfull flatness of $\Ohat_{\k^n, a}$ over $\k[x]_a$
one can prove that: $\hat{I_2} \cap \k[x]_a[s] = J_2$
and thus $\hat{I_3}=J_3$. This is another way to recover
the well-known equality between the local and the formal
Bernstein-Sato ideals.
\end{itemize}
\end{rem}

We shall leave the proof of the second statement of this remark
since it is not the goal of our paper.

\section{Proofs}

In this section we give all the proofs. All the proofs, except
for Cor. \ref{cor:strat} and Cor. \ref{cor:glob-loc}, concern a
fixed point $a$; so for these proofs, we shall assume that $a=0$.

First let us prove lemma \ref{lem}.

\begin{proof}
Let us do the proof for $\mathrm{ann}_{\Dhat_0\langle t,
\ddt \rangle}(f^s)$. The other cases are similar.
Recall that $\mathrm{ann}_{\Dhat_0 \langle t, \ddt
\rangle}(f^s)$ is the left ideal
$\{P \in \Dhat_0\langle t, \ddt \rangle \ | \ P\cdot f^s=0\}$.
Let $P$ be in this ideal. Modulo the elements in (\ref{eq:I})
we may assume that $P\in \k[[x]][\ddt]$. Let
us write $P=\sum_{\nu} c_{\nu}
\ddt^\nu$ with $\nu \in \N^p$ and
$\ddt^\nu=\prod_{j=1}^p \dtj^{\nu_j}$
and $c_\nu \in \k[[x]]$. Then
\[
0=P f^s=\sum_\nu (-1)^{|\nu |} c_\nu
\prod_{j=1}^p (s_j \cdots (s_j-\nu_j+1) f_j^{-\nu_j}) f^s.
\]
This equality takes place in the free module
$\k[[x]][1/F, s] \cdot f^s$, thus all the terms
in the sum are zero, which implies that all the $c_\nu$
are zero. This concludes the proof.
\end{proof}

Now let us prove the propositions.
Let us begin with the
\begin{proof}[Proof of Proposition \ref{p1}]
Let us prove the second equality (the proof is the same for
the first one).
Let $b(s)$ be in $\k[s]$. If $b(s) \in \B_{loc,0}(f)$ then
$b(s) f^s = P f^{s+1}$ for some $P \in D_0[s]$. Thus
$b(s) -P F$ annihilates $f^s$, i.e. $b(s) \in (D_0[s] F
+J_1) \cap \k[s]=J_3$. The converse uses the same
arguments. We leave the details to the reader.
\end{proof}

Let us go on with the
\begin{proof}[Proof of Proposition \ref{p2}]
We have $I \subset J$ so $I_1 \subset J_1$ and then
$D_0[s] I_1 \subset J_1$. Let us show the converse inclusion.
Take $P$ in $J_1=(D_0\langle t, \ddt \rangle \cdot I) \cap
D_0[s]$. Writing $P$ as an element in $D_0 \langle t, \ddt
\rangle I$ and as an element of $D_0[s]$ we may clear the
denominators and obtain the existence of $c(x) \in \k[x]$
with $c(0) \ne 0$ such that $c(x) P \in I \cap D[s]$.
Thus $P$ is in $D_0[s] (I \cap D[s])=D_0[s] I_1$.
This ends the proof for the first equality.
For the second one the arguments are exactly the same.
\end{proof}

Proposition \ref{p3} is an obvious consequence
of the following lemma.

\begin{lem}\label{lem:p3}\
\begin{itemize}
\item[(i)]
If $\Up \subset \k[x,s]$ is an ideal with
$0 \notin V(\Up \cap \k[x])$ then

$ (\k[x]_{0}[s] \cdot \Up) \cap \k[s]= \k[s]$.
\item[(ii)]
If $\Up \subset \k[x,s]$ is a primary ideal with
$0 \in V(\Up \cap \k[x])$ then

$ (\k[x]_{0}[s] \cdot \Up) \cap \k[s]=
\Up \cap \k[s]$.
\item[(iii)]
Given ideals $\Up_1, \ldots, \Up_r$ in $\k[x,s]$, we have:

$ \k[x]_0[s] \cdot (\bigcap_{i=1}^r \Up_i)=
\bigcap_{i=1}^r (\k[x]_0[s] \cdot \Up_i)$.
\end{itemize}
\end{lem}

\begin{proof}
\begin{itemize}
\item[(i)]
If $0 \notin V(\Up \cap \k[x])$ then there exists
$g \in \Up \cap \k[x]$ such that $g(0) \ne 0$ which
implies that $1 \in \Up$. 
\item[(ii)]
Let $f \in (\k[x]_0[s] \cdot \Up) \cap \k[s]$,
then there exists $c \in \k[x]$ with $c(0) \ne 0$
such that $c f \in \Up$. Assume, by contradiction,
that $f \notin \Up$. Then since $\Up$ is primary,
$c^l \in \Up$ for some $l \in \N$. This implies
$c(0)=0$: contradiction. Thus $f \in \Up \cap \k[s]$.
We proved the left-right inclusion. The reverse one
is trivial.
\item[(iii)]
Since the left-right inclusion is trivial, let us prove
the other one. Let $f$ be in $\bigcap_{i=1}^r
(\k[x]_0[s] \cdot \Up_i)$. Then for each $i$,
$c_i f \in \Up_i$ for some $c_i \in \k[x]$ satisfying
$c_i(0) \ne 0$. As a consequence, $(\prod_1^r c_i)
f \in \bigcap_1^r \Up_i$ and then $f \in \k[x]_0[s]
\cdot (\bigcap_{i=1}^r \Up_i)$.
\end{itemize}
\end{proof}

Now, let us work with arbitrary points $a \in \k^n$ and
prove the two corollaries.

\begin{proof}[Proof of Corollary \ref{cor:strat}]
First, it is clear that each $W_\sigma$ is locally closed.
Moreover, it is clear that any $a \in \k^n$ belongs to some
$W_\sigma$ (indeed, $a \in W_{\sigma_a}$ with the notations
of Prop. \ref{p3}).
Thus we have a constructible stratification of $\k^n$.
The constancy of the map $(a \mapsto \B_{loc,a}(f))$ on each
$W_\sigma$ follows from the obvious observation that if
$a$ and $a'$ are two points in a $W_\sigma$ then
$\sigma_a=\sigma_{a'}$, which implies, by using the whole
algorithm and in particular Prop. \ref{p3}, that $\B_{loc,a}(f)=
\B_{loc,a'}(f)$.
\end{proof}

\begin{proof}[Proof of Corollary \ref{cor:glob-loc}]
First, it is obvious from the definitions that
$\B_{glob}(f)$ is included in any $\B_{loc,a}(f)$ so we
have the inclusion: $\B_{glob}(f) \subset
\cap_a \B_{loc,a}(f)$. Let us prove the converse one.
We follow the notations of Prop. \ref{p3} and Cor.~\ref{cor:strat}.
Let us fix $i \in \{1, \ldots, r\}$.
Notice that since $\Up_i \subset \k[x,s]$ is primary,
$\Up_i \cap \k[x]$ is also primary in $\k[x]$ so
$V(\Up_i \cap \k[x])$ is irreducible.
Let
\[\tau_i=\{ k \in \{1, \ldots, r\} |
V(\Up_i \cap \k[x]) \subset V(\Up_k \cap \k[x]) \}.\]

Assume, by contradiction, that $W_{\tau_i}=\emptyset$. Then
$V(\Up_i \cap \k[x]) \subset \cup_{k \notin \tau_i} V(\Up_k \cap \k[x])$
and by irreducibility of $V(\Up_i \cap \k[x])$ it would be contained
in some $V(\Up_k \cap \k[x])$ for some $k \notin \tau_i$, which
is impossible. So let $a_i \in W_{\tau_i}$. Then we have:
$\B_{loc,a_i}(f) \subset \Up_i \cap \k[s]$.
As a consequence, we get:
\begin{eqnarray*}
\bigcap_{a \in \k^n} \B_{loc,a}(f) & \subset &
\bigcap_{i=1,\ldots, r} \B_{loc, a_i}(f) \\
& \subset & \bigcap_{i=1,\ldots,r} (\Up_i \cap \k[s])\\
& = & I_2 \cap \k[s]\\
& = & \B_{glob}(f).
\end{eqnarray*}
\end{proof}

Let us end this section with the proof of Proposition~\ref{p4}
(again, we assume that $a=0$).

%%% proof of Proposition 1.7
%%% T.Oaku, April 24, 2008

\begin{proof}
We shall prove only the first statement. The arguments
are the same for the second equality. Only one inclusion
is non trivial, namely
$\mathrm{ann}_{\Dhat_0[s]}(f^s) \subset \Dhat_0[s] I_1$.
First we have a natural isomorphism
\begin{equation}\label{seq1}
\k[[x]] \otimes_{\k[x]} D[s] f^s
\simeq
\Dhat_0[s] \otimes_{D[s]} D[s] f^s.
\end{equation}
This gives a natural left $\Dhat_0[s]$-module structure on
the tensor product of the left-hand side.

Now let us start with the following exact sequence
of $D[s]$-modules:
\[
0 \rightarrow I_1 \rightarrow D[s] \rightarrow D[s] f^s
\rightarrow 0.
\]
By flatness of $\Dhat_0[s]$ over $D[s]$, we get an exact
sequence of $\Dhat_0[s]$-modules:
%\begin{equation}\label{seq2}
\[
0 \rightarrow \Dhat_0[s] I_1 \rightarrow \Dhat_0[s] \rightarrow
\Dhat_0[s] \otimes D[s] f^s \rightarrow 0.
\]
%\end{equation}

\noindent
Thanks to the isomorphism (\ref{seq1}),
it remains to prove that $\k[[x]] \otimes D[s] f^s$
is naturally isomorphic to $\Dhat_0[s] f^s$.

We have an injective $D[s]$-morphism:
\[0 \to D[s] f^s \to \k[x][1/F,s] f^s.\]
Flatness of $\k[[x]]$ over $\k[x]$ implies the exactness of
%\begin{equation}\label{seq3}
\[
0 \to \k[[x]] \otimes_{\k[x]} D[s] f^s \stackrel{\varphi}{\to}
\k[[x]] \otimes_{\k[x]} \k[x][1/F,s]f^s. 
\]
%\end{equation}
On the other hand, there is a natural homomorphism
\[
\k[[x]] \otimes_{\k[x]} \k[x][1/F,s]f^s \stackrel{\psi}{\to} 
\k[[x]][1/F,s]f^s. 
\]
An arbitrary element of $\k[[x]] \otimes_{\k[x]} \k[x][1/F,s]f^s$ is
written in the form $\sum_\mu \hat{c}_\mu(x)\otimes s^\mu F^{-m}f^s$
with $\hat{c}_\mu(x) \in \k[[x]]$ and $m \in \N$, and it is sent to
$\sum_\mu \hat{c}_\mu(x)s^\mu F^{-m}f^s$ by $\psi$. 
The latter is zero if and only if $\hat{c}_\mu(x) = 0$ for any $\mu$. 
This shows that $\psi$ is an isomorphism. 

An arbitrary element 
$\sum_{\mu,\beta} \hat{c}_{\mu,\beta}(x)\otimes s^\mu \partial_x^\beta f^s$ of 
$\k[[x]] \otimes_{\k[x]} D[s] f^s$ with $\hat{c}_{\mu,\beta}(x) \in \k[[x]]$
is sent to 
$\sum_{\mu,\beta} \hat{c}_{\mu,\beta}(x) s^\mu \partial_x^\beta f^s$
by $\psi\circ\varphi$.
This implies that the image of $\psi\circ\varphi$ coincides with 
$\Dhat_0[s]f^s$. 

Thus $\psi\circ\varphi$ gives a natural isomorphism $\k[[x]] \otimes
D[s] f^s \simeq \Dhat_0[s] f^s$ and it is naturally
$\Dhat_0[s]$-linear.
Hence we get an exact sequence of $\Dhat_0[s]$-modules
\[0 \rightarrow \Dhat_0[s] I_1 \rightarrow \Dhat_0[s]
\rightarrow \Dhat_0[s] f^s \rightarrow 0\]
with natural maps. This completes the proof.
\end{proof}

%%%%%%%%%%%%%%%%%%%%%%%%

\section{Examples}

Here let us start with a result ``well-known to specialists''.
Let us assume $f$ to be in $\k[[x]]^p$.

\begin{lem} \label{lem:classic}
\begin{itemize}
\item[(i)]
Let $u_1, \ldots, u_p$ be units in $\k[[x]]$. Then
$\B_{\k[[x]]}(f)$ is equal to
$\B_{\k[[x]]}(u_1f_1, \ldots, u_p f_p)$.
\item[(ii)]
If $f_{k+1}, \ldots, f_p$ are units in $\k[[x]]$ then
$\B_{\k[[x]]}(f)$ is equal to $\k[s_1, \ldots, s_p] \cdot
\B_{\k[[x]]}(f_1, \ldots, f_k)$.
\item[(iii)]
Let $\K \supset \k$ be a field extension of $\k$. Then
$\B_{\K[x]}(f)=\K[s] \cdot \B_{\k[x]}(f)$ and for any $a \in \k^n$,
$\B_{\K[[x-a]]}(f)=\K[s] \cdot \B_{\k[[x-a]]}(f)$.
\item[(iv)]
Suppose $f \in \k[x]^p$. Let $a \in \k^n$ be such that $f(a)=0$
and $f$ is smooth at $a$.
Then $\B_{loc, a}(f)$ is generated by $\prod_{j=1}^p (s_j+1)$.
\end{itemize}
\end{lem}

\begin{proof}
\begin{itemize}
\item[(i)]
Let us sketch the proof. One can see that it is enough to
make the proof with $f'=(u_1f_1, f_2, \ldots, f_p)$.
By considering the change of variables
$(x,t) \mapsto (x, u^{-1} t_1, t_2, \ldots, t_p)$ one can
prove that the annihilators of $f^s$ and ${f'}^s$ coincide
in $\Dhat_0 \langle t, \ddt \rangle$ from which one
can conclude.
\item[(ii)]
Thanks to (i), we may assume that $f_{k+1}=\cdots=f_p=1$.
Moreover it is enough to make the proof with $k=p-1$.
Set $f'=(f_1, \ldots, f_{p-1})$. The inclusion
$\k[s]\B(f') \subset \B(f)$ is trivial.
Let us prove the converse one.
Set $s'=(s_1,\ldots,s_{p-1})$ and let
$b(s',s_p) \in \B(f)$. We have:
\[b(s',s_p) {f'}^{s'} 1^{s_p} \in
\Dhat_0[s', s_p] f_1 \cdots f_{p-1} {f'}^{s'} 1^{s_p}.\]
Thus we see that for any $\lambda \in \k$, $b(s', \lambda)$
belongs to $\B(f')$. Let us write: $b(s', s_p)=
\sum_{l=0}^d c_k(s') s_p^l$. Let $\lambda_0, \ldots, \lambda_d$
be pairwise distinct elements in $\k$. We thus obtain the
existence of $b_0(s'), \ldots, b_d(s') \in \B(f')$ such that
\[
\left(
\begin{array}{cccc}
1 & \lambda_0 & \cdots & \lambda_0^d \\
\vdots &   &   &  \vdots \\
1 & \lambda_d & \cdots & \lambda_d^d
\end{array}
\right)
\left(
\begin{array}{c}
c_0(s') \\
\vdots \\
c_d(s')
\end{array}
\right)
=
\left(
\begin{array}{c}
b_0(s') \\
\vdots \\
b_d(s')
\end{array}
\right).
\]
This is an invertible Vandermonde matrix from which we deduce
that each $c_l$ is in $\B(f')$. This implies $b(s', s_p) \in
\k[s] \B(f')$.
\item[(iii)]
See e.g. Brian\c{c}on, Maisonobe \cite[Prop. 1.5]{bm02}.
\item[(iv)]
Let $C \subset \k$ be the (finite) set of the coefficients
of the $f_j$'s and let $\K=\Q(C \cup \{a\})$. Then
$\K$ is identified with a subfield of $\C$. So, thanks
to (iii), $\B_{loc,a}(f)=\B_{\C[x]_a}(f)$ and it is well-known
that it is equal to $\B_{\C\{x-a\}}(f)$. Thus we can conclude
with Brian\c{c}on-Maynadier \cite[Prop. 1.2]{bmay}. 
\end{itemize}
\end{proof}

\noindent
In the following, $\langle G \rangle$ denotes the
ideal generated by the set $G$.

The computations were made using Kan/sm1 \cite{kan}
and Risa/Asir \cite{asir}. Moreover the
computations were made over the field $\Q$.
In the last paragraph of this section, we check that
the results are valid over $\C$.

\subsection{Example 1}

This first example is trivial in the sense that all the
local Bernstein-Sato ideals can be computed using the previous
lemma. Let us define $f \in \Q[x,y]^3$ by:
\[(f_1,f_2,f_3)(x,y)=(x , y , 1-x-y).\]
Only with the lemma above one can say that given
$a \in \Q^2$, $\B_{loc,a}(f)$ is equal to:
\begin{itemize}
\item
$\Q[s]=\Q[s_1,s_2,s_3]$ for $a \notin \{x=0\} \cup \{y=0\}
\cup \{x+y=1\}$,
\item
$ \langle (s_1+1) \rangle$ for $a \in \{x=0\} \smallsetminus
\{(0,0), (0,1)\}$,
\item
$\langle (s_2+1)\rangle$ for $a \in \{y=0\} \smallsetminus
\{(0,0), (1,0)\}$,
\item
$\langle (s_3+1)\rangle$ for $a \in \{x+y=1\}
\smallsetminus \{(0,1), (1,0)\}$,
\item
$\langle (s_1+1)(s_2+1)\rangle$ if $a=(0,0)$,
$\langle (s_1+1)(s_3+1)\rangle$ if $a=(0,1)$,
$\langle (s_2+1)(s_3+1)\rangle$ if $a=(1,0)$.
\end{itemize}
By using Cor. \ref{cor:glob-loc} one has:

$\B_{glob}(f)=\langle(s_1+1)(s_2+1)(s_3+1) \rangle$.\\
We notice that the global Bernstein-Sato ideal is different
from all the local ones.
By using our algorithm, we found the following primary
decomposition for $I_2 \subset \Q[x,y,s_1,s_2,s_3]$:
\[I_2=\Up_1 \cap \Up_2 \cap \Up_3 \text{ where }
\Up_j=\langle s_j+1, f_j \rangle\]
which obviously enables to recover all the results above.

\subsection{Example 2}

Here, $f \in \Q[x,y]^3$ is given by:
\[(f_1,f_2,f_3)(x,y)= (y , y-2x+1 , y-x^2).\]
A computation by hand of all the local Bernstein-Sato
ideals is far from being easy. The computed primary decomposition
of the ideal $I_2$ has seven primary components $\Up_i$.
For each of them, we present:
$\sqrt{\Up_i}$, $\sqrt{\Up_i \cap \Q[x,y]}$,
$\Up_i \cap \Q[s_1,s_2,s_3]$.
\[
\begin{array}{llll}
(1) &
\langle  s_1+1, y\rangle &
\langle y \rangle &
\langle s_1+1 \rangle \\

(2) &
\langle  s_2+1, y-2x+1 \rangle &
\langle y-2x+1 \rangle &
\langle s_2+1 \rangle \\

(3) &
\langle s_3+1, y-x^2 \rangle &
\langle y -x^2 \rangle &
\langle s_3+1 \rangle \\

(4) &
\langle 2s_1+2s_3+3, x,y \rangle &
\langle x,y \rangle &
\langle 2s_1+2s_3+3 \rangle \\

(5) &
\langle 2s_2+2s_3+3, x-1, y-1 \rangle &
\langle x-1, y-1 \rangle &
\langle 2s_2+2s_3+3 \rangle \\

(6) &
\langle 2s_1+2s_3+5, x, y \rangle &
\langle x,y \rangle &
\langle 2s_1+2s_3+5 \rangle \\

(7) &
\langle 2s_2+2s_3+5, x-1, y-1 \rangle &
\langle x-1, y-1 \rangle &
\langle 2s_2+2s_3+5 \rangle \\
\end{array}
\]
We get:
$\B_{glob}(f)=\langle  (s_1+1) (s_2+1) (s_3+1) (2s_1+2s_3+3)
(2s_1+2s_3+5) (2s_2+2s_3+3) (2s_2+2s_3+5) \rangle$. This ideal
is different from all the local Bernstein-Sato ideals.

\subsection{Example 3}

Here, $f\in \Q[x,y]^2$ is given as follows:
\[(f_1,f_2)=(x^3+y^2, x^2+y^3).\]
This example has been stimulating several authors for several
years (in the global case and in the local case see
e.g. Bahloul \cite{b-JSC, b-kyushu},
Castro-Jim\'enez and Ucha-Enr\'{\i}quez \cite[\S 3.4]{castro-ucha}
and Gago-Vargas et al. \cite[Rem. 2]{gago}).

The obtained primary decomposition of the ideal $I_2$ is made of
twelve primary components $\Up_i$. In the following we present,
for each $i$, the ideals $\sqrt{\Up_i}$,
$\sqrt{\Up_i \cap \Q[x,y]}$ and $\Up_i \cap \Q[s_1, s_2]$.
\[
\begin{array}{llll}
(1) &
\langle x^3+y^2 , s_1+1 \rangle &
\langle x^3+y^2 \rangle &
\langle s_1+1 \rangle \\

(2) &
\langle x^2+y^3 , s_2+1 \rangle &
\langle x^2+y^3 \rangle &
\langle s_2+1 \rangle \\

(3) &
\langle x,y , 4s_1+6s_2+5 \rangle &
\langle x,y \rangle &
\langle 4s_1+6s_2+5 \rangle\\

(4) &
\langle x,y , 4s_1+6s_2+7 \rangle &
\langle x,y \rangle &
\langle 4s_1+6s_2+7 \rangle \\

(5) &
\langle x,y , 4s_1+6s_2+9 \rangle &
\langle x,y \rangle &
\langle 4s_1+6s_2+9 \rangle \\

(6) &
\langle x,y , 4s_1+6s_2+11 \rangle &
\langle x,y \rangle &
\langle 4s_1+6s_2+11 \rangle \\

(7) &
\langle x,y , 4s_1+6s_2+13 \rangle &
\langle x,y \rangle &
\langle 4s_1+6s_2+13 \rangle \\

(8) &
\langle x,y , 6s_1+4s_2+5 \rangle &
\langle x,y \rangle &
\langle 6s_1+4s_2+5 \rangle \\

(9) &
\langle x,y , 6s_1+4s_2+7 \rangle &
\langle x,y \rangle &
\langle 6s_1+4s_2+7 \rangle \\

(10) &
\langle x,y , 6s_1+4s_2+9 \rangle &
\langle x,y \rangle &
\langle 6s_1+4s_2+9 \rangle \\

(11) &
\langle x,y, 6s_1+4s_2+11 \rangle &
\langle x,y \rangle &
\langle 6s_1+4s_2+11 \rangle \\

(12) &
\langle x,y , 6s_1+4s_2+13 \rangle &
\langle x,y \rangle &
\langle 6s_1+4s_2+13 \rangle \\
\end{array}
\]

We get that $\B_{glob}(f)$ coincides with $\B_{loc,0}(f)$ and
is generated by the following polynomial:

\noindent
$(s_1+1) (s_2+1)
(4s_1+6s_2+5)(4s_1+6s_2+7)(4s_1+6s_2+9)(4s_1+6s_2+11)(4s_1+6s_2+13)
(6s_1+4s_2+5)(6s_1+4s_2+7)(6s_1+4s_2+9)(6s_1+4s_2+11)(6s_1+4s_2+13)$.

\subsection{Example 4}

Here $f \in \Q[x,y,z]^2$ is given by:
\[(f_1,f_2)=(z, x^4 + y^4 + 2 z x^2 y^2).\]
This important example is taken from
Brian\c{c}on, Maynadier \cite{bmay} where the authors proved
for the first time that local Bernstein-Sato ideals are not
principal in general.
They did not calculate $\B_{loc,0}(f)$; they only proved that it is
not principal.

The computation of a primary decomposition of $I_2$ outputs
nine primary components $\Up_i$ and
as before, we give, for each $i$:
$\sqrt{\Up_i}$,
$\sqrt{\Up_i \cap \Q[x,y,z]}$ and $\Up_i \cap \Q[s_1, s_2]$.
\[
\begin{array}{llll}
(1) &
\langle z , s_1+1 \rangle &
\langle z \rangle &
\langle s_1+1 \rangle \\

(2) &
\langle f_2 , s_2+1 \rangle &
\langle f_2 \rangle &
\langle s_2+1 \rangle \\

(3) &
\langle x,y , s_2+1 \rangle &
\langle x,y \rangle &
\langle (s_2+1)^2 \rangle \\

(4) &
\langle x,y , 2s_2+1 \rangle &
\langle x,y \rangle &
\langle 2s_2+1 \rangle \\

(5) &
\langle x,y , 4s_2+3 \rangle &
\langle x,y \rangle &
\langle 4s_2+3 \rangle \\

(6) &
\langle x,y , 4s_2+5 \rangle &
\langle x,y \rangle &
\langle 4s_2+5 \rangle \\

(7) &
\langle x,y,z , s_1+2, 2s_2+3 \rangle &
\langle x,y,z \rangle &
\langle s_1+2, 2s_2+3 \rangle \\

(8) &
\langle x,y, z-1 , 2s_2+3 \rangle &
\langle x,y, z-1 \rangle &
\langle 2s_2+3 \rangle \\

(9) &
\langle x,y,z+1 , 2s_2+3 \rangle &
\langle x,y,z+1 \rangle &
\langle 2s_2+3 \rangle
\end{array}
\]

As a consequence, $\B_{loc,0}(f)$ is generated by two elements:
\begin{eqnarray*}
\B_{loc,0}(f) &= &
\langle (s_1+1) (s_2+1)^2 (2s_2+1) (4s_2+3) (4s_2+5) (s_1+2),\\
 & & (s_1+1) (s_2+1)^2 (2s_2+1) (4s_2+3) (4s_2+5) (2s_2+3) \rangle
\end{eqnarray*}
and $\B_{glob}(f)$ is principal:
\[\B_{glob}(f)=\langle
(s_1+1) (s_2+1)^2 (2s_2+1) (2s_2+3) (4s_2+3) (4s_2+5)
\rangle\]

\subsection{Example 5}

Here $f \in \Q[x,y,z]^2$ is given by:
\[(f_1,f_2)=(z, x^5 + y^5 + z x^2 y^3).\]
The computed primary decomposition of $I_2$ is made of
twelve terms $\Up_i$. Again, for each $i=1$, we give
$\sqrt{\Up_i}$,
$\sqrt{\Up_i \cap \Q[x,y,z]}$ and $\Up_i \cap \Q[s_1, s_2]$.
\[
\begin{array}{llll}
(1) &
\langle z , s_1+1 \rangle &
\langle z \rangle &
\langle s_1+1 \rangle \\

(2) &
\langle f_2 , s_2+1 \rangle &
\langle f_2 \rangle &
\langle s_2+1 \rangle \\

(3) &
\langle x,y , s_2+1 \rangle &
\langle x,y \rangle &
\langle (s_2+1)^2 \rangle \\

(4) &
\langle x,y , 5s_2+2 \rangle &
\langle x,y \rangle &
\langle 5s_2+2 \rangle \\

(5) &
\langle x,y , 5s_2+3 \rangle &
\langle x,y \rangle &
\langle 5s_2+3 \rangle \\

(6) &
\langle x,y , 5s_2+4 \rangle &
\langle x,y \rangle &
\langle 5s_2+4 \rangle \\

(7) &
\langle x,y , 5s_2+6 \rangle &
\langle x,y \rangle &
\langle 5s_2+6 \rangle \\

(8) &
\langle x,y,z , s_1+2, 5s_2+7 \rangle &
\langle x,y,z \rangle &
\langle s_1+2, 5s_2+7 \rangle \\

(9) &
\langle x,y,z , s_1+3, 5s_2+7 \rangle &
\langle x,y,z \rangle &
\langle s_1+3, 5s_2+7 \rangle \\

(10) &
\langle x,y,z , s_1+4, 5s_2+7 \rangle &
\langle x,y,z \rangle &
\langle s_1+4, 5s_2+7 \rangle \\

(11) &
\langle x,y,z , s_1+5, 5s_2+7 \rangle &
\langle x,y,z \rangle &
\langle s_1+5, 5s_2+7 \rangle \\

(12) &
\langle x,y,z , s_1+2, 5s_2+8 \rangle &
\langle x,y,z \rangle &
\langle s_1+2, 5s_2+8 \rangle
\end{array}
\]

We get that $\B_{glob}(f)$ and $\B_{loc,0}(f)$ are equal and generated
by these three elements:

$(s_1+1)(s_2+1)^2 (5s_2+2) (5s_2+3) (5s_2+4) (5s_2+6)
(s_1+2) (s_1+3) (s_1+4) (s_1+5)$,

$(s_1+1)(s_2+1)^2 (5s_2+2) (5s_2+3) (5s_2+4) (5s_2+6)
(5s_2+7) (s_1+2)$,

$(s_1+1)(s_2+1)^2 (5s_2+2) (5s_2+3) (5s_2+4) (5s_2+6)
(5s_2+7) (5s_2+8)$.

\subsection{Validity of the computations over $\C$}

First, let us state some general results useful in the sequel.

Let $\k \subseteq \K$ be two fields of characteristic zero.

\begin{lem}\label{lem:ex_elimin}
Let $J$ be an ideal in $\k[y,z]=\k[y_1,
\ldots, y_q, z_1, \ldots, z_r]$ then
\[
(\K[y,z] \cdot J) \cap \K[y] = \K[y] \cdot (J \cap \k[y]).
\]
\end{lem}

\begin{proof}
Let us consider $g \in (\K[y,z] \cdot J) \cap \K[y]$.
Let $f_1, \ldots, f_s$ be a system of generators of $J$
and let us write $g=\sum_j u_j f_j$ with $u_j \in \K[y,z]$.
Let $e_l$, $l\in L$, be a basis of the $\k$-linear space
generated by all the coefficients of the $u_j$'s.
So one can write $g \in \oplus_l J e_l \subset \oplus_l
\k[x,y] e_l$. Since $g \in \K[y]$, we obtain
$g\in \oplus (J \cap \k[y]) e_l$.
The left-right inclusion is proven. The other one being
trivial, the proof is complete.
\end{proof}

\begin{lem}\label{lem:ex_primarity-K-k}
Let $J$ be an ideal in $\k[x]=\k[x_1,\ldots,x_n]$.
\begin{enumerate}
\item
$(\K[x] J) \cap \k[x]=J$.
\item
If $\K[x] J$ is primary in $\K[x]$ then $J$
is primary in $\k[x]$.
\end{enumerate}
\end{lem}

\begin{proof}
The right-left inclusion of (1) is trivial while the
left-right one can be proven by arguments similar to those
used for proving lemma \ref{lem:ex_elimin}.
Claim (2) is a direct consequence of the definitions and (1).
\end{proof}

We shall denote by $\Gal$ the Galois group of the extension
$\K / \k$. If $\tau \in \Gal$ then we shall also denote
by $\tau$ the induced ring automorphism of
$\K[x]=\K[x_1,\ldots,x_n]$.

\begin{lem}\label{lem:ex_Galois_fixed}
Assume that $\K / \k$ is a Galois extension.
Let $J \subset \K[x]=\K[x_1,\ldots,x_n]$ be an ideal.
Suppose that for any $\tau \in \Gal$, $\tau(J) \subset J$.
Then there exists an ideal $J_0 \subset \k[x]$
such that $J=\K[x] J_0$.
\end{lem}

\begin{proof}
Given $\tau \in \Gal$, we have $\tau^{-1}(J) \subset J$ from
which we deduce that $\tau(J)=J$.

Let $G$ be the reduced Gr\"obner basis of $J$ with
respect to a fixed well-order. In view of Buchberger's criterion,
we get that $\tau(G)$ is also the reduced Gr\"obner
basis of $J$ for any $\tau\in \Gal$.
Therefore, for any $g \in G$, $\tau(g)=g$.
Thus, $\tau$ fixes each coefficient of $g$. Since $\tau$ is
arbitrary and the extension $\K/\k$ is Galois, we get
$g \in \k[x]$.
\end{proof}

\begin{prop}\label{prop:ex_main}
Let $\Upsilon$ be a primary ideal of $\k[x]= \k[x_1,\dots,x_n]$ 
and $\K$ be a field containing the algebraic closure of $\k$.
If the radical of $\K[x]\Upsilon$ is a prime ideal of $\K[x]$, 
then $\K[x]\Upsilon$ is a primary ideal of $\K[x]$.
\end{prop}

\begin{proof}
Take an irredundant primary decomposition 
\begin{equation}\label{eq:minimal}
 \K[x]\Up = \bigcap_{j=0}^m \Up^{\K}_j
\end{equation}
in $\K[x]$. We assume by contradiction that $m\ge 1$.

The algorithms for computing a primary decomposition imply
that there exists a finite Galois extension $\K'$ of $\k$
such that $\Upsilon^{\K}_j$ are defined over $\K'$
(see e.g. \cite[Chap. 4]{singular}).
Then the field extension from $\K'$ to $\K$ is trivial
in the sense that the primarity of each component is preserved.
Thus we may now assume $\K' = \K$.

Since the radicals $\sqrt{\Upsilon^{\K}_j}$ are distinct
and $\sqrt{\K[x]\Upsilon}$ is prime, we may assume,
without loss of generality that
$\sqrt{\Upsilon^{\K}_0} = \sqrt{\K[x]\Upsilon}$ and
the dimension of $\sqrt{\Upsilon^{\K}_j}$ is less than that of
$\sqrt{\K[x]\Upsilon}$ for $j = 1,\dots,m$.

Let $\tau$ be an element of the Galois group
$\mathrm{Gal}(\K/\k)$. Then
\begin{equation}\label{eq:tau}
 \K[x]\Upsilon = \bigcap_{j=0}^m \tau(\Upsilon^{\K}_j)
\end{equation}
is also an irredundant primary decomposition.
Since the non-embedded primary components are unique, 
we have $\tau(\Upsilon^{\K}_0) = \Upsilon^{\K}_0$.  
Since $\tau \in \Gal$ is arbitrary,
this implies, by lemma \ref{lem:ex_Galois_fixed},
that $\Upsilon^{\K}_0$ is defined over $\k$, 
i.e., there is an ideal $\Upsilon_0$ of $\k[x]$ 
such that $\Upsilon^{\K}_0 = \K[x]\Upsilon_0$.
By lemma \ref{lem:ex_primarity-K-k}(2), $\Up_0$ is
primary in $\k[x]$.

For each $j = 1,\dots,m$, by lemma \ref{lem:ex_Galois_fixed},
there exists an ideal $\Upsilon_j$ of 
$\k[x]$ such that
\begin{equation}\label{eq:Gal}
\K[x]\Upsilon_j = \bigcap_{\tau \in \mathrm{Gal}(\K/\k)} 
\tau(\Upsilon^{\K}_j).
\end{equation}
Moreover, $\Upsilon_j$ is primary in $\k[x]$.
Indeed, assume $f,g \in \k[x]$ satisfy $fg \in \Upsilon_j$ 
and $f \not\in \Upsilon_j$.  Then there exists $\tau \in \Gal$
such that $f \not\in \tau(\Upsilon^{\K}_j)$.  
Since $f$ is fixed by every element of $\Gal$, it follows that
$f \notin \tau(\Upsilon^{\K}_j)$ for any
$\tau \in \mathrm{Gal}(\K/\k)$.
In particular, $f \notin \Up_j^\K$. Hence there exists
an integer $\nu$ such that $g^\nu \in \Up_j^\K$. Since
$g^\nu=\tau(g^\nu) \in \tau(\Up_j^\K)$ for any $\tau \in \Gal$,
we have $g^\nu \in \K[x]\Up_j$, which implies, by
lemma~\ref{lem:ex_primarity-K-k}(1), $g^\nu \in \Upsilon_j$.

Combining equalities (\ref{eq:Gal}) and (\ref{eq:tau}), we get
\[
\K[x] \Up= \bigcap_{j=0}^m (\K[x] \Up_j).
\]
Using lemma \ref{lem:ex_primarity-K-k}(1), we obtain a
(not necessarily irredundant) primary decomposition in $\k[x]$:
\[
 \Upsilon = \bigcap_{j=0}^m \Upsilon_j.
\]

Since $\Up$ is primary and $\dim(\Up)=\dim(\Up_0)>
\dim(\Up_j)$ for $j\ge 1$, the uniqueness of the number of
components in irredundant primary decomposition implies that
$\Up_0 \subset \Up_j$ for $j=1, \ldots, m$. Hence we get
\[
\Up^\K_0 = \K[x] \Up_0 \subset \K[x] \Up_j=
\bigcap_{\tau \in \Gal} \tau(\Up_j^\K) \subset \Up_j^\K
\]
for $j=1, \ldots, m$. This contradicts the irredundancy
of (\ref{eq:minimal}).
\end{proof}

Let us return to our problem. We consider
the general situation: $f\in \Q[x]^p$ with $x=(x_1,\ldots,x_n)$.
The ideals $I$, $I_1$, $I_2$ introduced after lemma \ref{lem}
are defined over $\Q$.
Considering $f$ in $\C[x]^p$ one can define the ideals
$I^\C$, $I_1^\C$ and $I_2^\C$ over $\C$.

\begin{lem} \label{lem:ex_QvsC}\
\begin{enumerate}
\item
$I^\C= D_{\C[x]} \langle t, \ddt \rangle \cdot I$,
\item
$I_1^\C = D_{\C[x]} [s] \cdot I_1$,
\item
$I_2^\C = \C[x,s] \cdot I_2$.
\end{enumerate}
\end{lem}
Thus the construction is the same over $\Q$ and over $\C$
until step 2.

\begin{proof}
Statement (1) is a consequence of lemma \ref{lem}.
Statements (2) and (3) are proven exactly as
lemma~\ref{lem:ex_elimin}.
\end{proof}

Let $I_2=\Up_1 \cap \cdots \cap \Up_r$ be a
primary decomposition of $I_2$ (over $\Q$).

\begin{clai*}
In examples 1 to 5, the obtained primary decomposition
of $I_2$ is primary over $\C$.
\end{clai*}

Assume for the moment that this claim is proven, i.e.
the decomposition
$\C[x,s]  I_2= \bigcap_{i=1}^r (\C[x,s]  \Up_i)$ is
primary. Then, applying Proposition \ref{p3}, we get
the following conclusion: for any $\alpha \in \C^n$,
\begin{eqnarray*}
\B_{loc,\alpha}(f) & = & \bigcap_{i \in \sigma_\alpha}
  ((\C[x,s] \Up_i) \cap \C[s])\\
& = & \C[s] \cdot (\bigcap_{i \in \sigma_\alpha} (\Up_i \cap \Q[s]))
  \text{ by lemma } \ref{lem:ex_elimin},
\end{eqnarray*}
where $\sigma_\alpha=\{i \ | \ 1 \le i \le r, \alpha \in
V((\C[x,s] \Up_i) \cap \C[x])\}$. Notice that by
lemma~\ref{lem:ex_elimin}, $(\C[x,s] \Up_i) \cap \C[x]=
\C[x]( \Up_i \cap \Q[x] )$.

Consequently the computations made over $\Q$ shall be valid
over $\C$. Thus to end this section, it remains to prove the
claim.

\begin{proof}[Proof of the claim]
Let $\Up=\Up_i$ be a primary component (over $\Q$) of the
computed decomposition of $I_2$.
By Proposition \ref{prop:ex_main}, it is enough to prove that
$\C[x,s] \sqrt{\Up}$ is prime.

In all the examples, $\sqrt{\Up}$
is either generated by polynomials of degree $1$
(i.e. defines an intersection of affine hyperplanes)
or of the form $\langle s_j+1, f_j \rangle$.
In the first case, $\C[x,s] \sqrt{\Up}$ is obviously prime.
It remains to analyse the second case. 

Assume for simplicity that $j=1$, then we have a ring isomorphism
$\C[x,s] / \langle s_1+1, f_1 \rangle \simeq
\C[x,s_2, \ldots,s_p]/ \langle f_1 \rangle$.
This relation shows that it is enough to prove that each
$f_j$ is irreducible over $\C$.

It is obviously the case in Ex. 1, 2 and 3 and for
$f_1$ in Ex. 4 and Ex. 5. Let us check
that $f_2$ in Ex. 4 is irreducible over $\C$. Suppose
that $f_2$ is a product of two polynomials. Regarding the
$z$-degree, we necessarily have a product of the form:
\[x^4 + y^4 +2 z x^2 y^2= u \cdot (v z + w),\]
where $u,v,w \in \C[x,y]$ but this is possible only if $u$ is
a constant.
Thus $f_2$ is irreducible. A similar proof shows
that $f_2$ in Ex. 5 is also irreducible over $\C$.
The proof is complete.
\end{proof}


\begin{thebibliography}{99}
% Eliminer les ref inutilisees ??

\bibitem[Bah01]{b-JSC}
R. Bahloul,
Algorithm for computing Bernstein-Sato ideals associated with
a polynomial mapping.
J. Symbolic Comput.  32  (2001),  no. 6, 643--662.


\bibitem[Bah05a]{b-compos}
R. Bahloul,
D\'emonstration constructive de l'existence de
polyn\^omes de Bernstein-Sato pour plusieurs
fonctions analytiques.
Compos. Math.  141  (2005),  no. 1, 175--191.


\bibitem[Bah05b]{b-kyushu}
R. Bahloul,
Construction d'un \'el\'ement remarquable de l'id\'eal
de Bernstein-Sato associ\'e \`a deux courbes planes analytiques,
Kyushu J. Math. 59 (2005),  no. 2, 421--441.


\bibitem[Ber72]{bernstein}
I. N. Bernstein,
Analytic continuation of generalized functions with respect to a
parameter.
Funkcional. Anal. i Prilo\v zen.  6  (1972), no. 4, 26--40.


\bibitem[Bjo74]{bjork}
J. E. Bj\"ork,
Dimensions of modules over algebras of differential operators.
Fonctions analytiques de plusieurs variables et analyse complexe
(Colloq. Internat. CNRS, No. 208, Paris, 1972),  pp. 6--11. ``Agora
Mathematica'', No. 1, Gauthier-Villars, Paris, 1974. 


%\bibitem[Bri95]{br95}
%J.~Brian\c{c}on,
%O\`u l'on pose quelques questions plus ou moins na\"{\i}ves
%sur les \'equations fonctionnelles g\'en\'eralis\'ees,
%unpublished notes , ``D-modules et Singularit\'es'' seminar,
%Nice, 1995.



\bibitem[B-M90]{bm90}
J. Brian\c{c}on and P. Maisonobe,
Examen de passage du local au global pour
les polyn\^omes de Bernstein-Sato,
unpublished notes, 1990.



\bibitem[B-M02]{bm02}
J. Brian\c{c}on and P. Maisonnobe,
Remarques sur l'id\'eal de Bernstein associ\'e \`a des polyn\^omes.
Preprint Universit\'e de Nice Sophia-Antipolis, no. 650, Mai 2002.



\bibitem[B-May99]{bmay}
J. Brian\c{c}on and H. Maynadier,
\'Equations fonctionnelles g\'en\'eralis\'ees : transversalit\'e
et principalit\'e de l'id\'eal de Bernstein-Sato,
J. Math. Kyoto Univ. 39 (1999), no. 2, 215--232.



\bibitem[CJ-UE04]{castro-ucha}
F. J. Castro-Jim\'enez, J. M. Ucha-Enr\'{\i}quez,
On the computation of Bernstein-Sato ideals,
J. Symbolic Comput. 37 (2004), no. 5, 629--639.



%\bibitem[Eis95]{eisenbud}
%D. Eisenbud,
%Commutative algebra. With a view toward algebraic geometry.
%Graduate Texts in Mathematics, 150.
%Springer-Verlag, New York, 1995.



%\bibitem[EHV92]{eisenbud_etal}
%D. Eisenbud, C. Huneke, W. Vasconcelos,
%Direct methods for primary decomposition,
%Invent. Math. 110 (1992), no. 2, 207--235.



\bibitem[GV-HH-UE05]{gago}
J. Gago-Vargas, M. I. Hartillo-Hermoso, J. M. Ucha-Enr\'{\i}quez,
Comparison of theoretical complexities of two methods
for computing annihilating ideals of polynomials,
J. Symbolic Comput. 40 (2005), no. 3, 1076--1086.



\bibitem[G-P02]{singular}
G.-M. Greuel, G. Pfister,
A Singular introduction to commutative algebra,
Springer-Verlag, Berlin, 2002.



\bibitem[Mal74]{malgrange}
B. Malgrange,
Le polyn\^ome de Bernstein d'une singularit\'e isol\'ee.
Lecture Notes in Math., pp. 98--119, Vol. 459, Springer, Berlin,
1975. 



\bibitem[M-NM91]{meb-nar}
Z. Mebkhout and L. Narv\`aez-Macarro,
La th\'eorie du polyn\^ome de Bernstein-Sato pour les
alg\`ebres de Tate et de Dwork-Monsky-Washnitzer.
Ann. Sci. \'Ecole Norm. Sup. (4)  24  (1991),  no. 2, 227--256.



\bibitem[Nak06]{nakayama}
H. Nakayama,
An algorithm computing local $b$-function by approximate
division algorithm in $\hat{\mathcal{D}}$,
preprint arXiv:math/0606437, 2006.



\bibitem[Nor]{asir}
Risa/Asir: an open source general computer algebra system,
Developed by Fujitsu Labs LTD, Kobe Distribution by Noro et al.,
see http://www.math.kobe-u.ac.jp/Asir/index.html.



\bibitem[Oak97]{oaku}
T. Oaku,
An algorithm of computing $b$-functions.
Duke Math. J.  87  (1997), no. 1, 115--132.



\bibitem[O-T99]{ot}
T. Oaku, N. Takayama,
An algorithm for de Rham cohomology groups of the complement of an
affine variety via D-module computation.
J. Pure Appl. Algebra 139 (1999), 201-233.



%\bibitem[Oak97]{oaku}
%T. Oaku,
%\emph{Algorithms for the $b$-function and $D$-modules associated with
%a polynomial},
%J. Pure Appl. Algebra 117/118 (1997), 495--518.


\bibitem[Sab87]{sabbah}
C. Sabbah,
Proximit\'e \'evanescente I. La structure polaire d'un
$\mathcal{D}$-Module,
Compositio Math.  62  (1987),  no. 3, 283--328;
Proximit\'e \'evanescente II. \'Equations fonctionnelles pour
plusieurs fonctions analytiques,
Compositio Math.  64  (1987),  no. 2, 213--241.



\bibitem[Tak91]{kan}
N. Takayama,
Kan/sm1: a system for computation in algebraic analysis, 1991-,
see http://www.math.kobe-u.ac.jp/KAN/index.html.

\end{thebibliography}
\end{document}